\documentclass[a4paper]{amsart}

\pdfoutput=1

\usepackage[english]{babel}
\usepackage[a4paper, inner=3cm,outer=3cm,top=3.3cm,bottom=3cm]{geometry}
\usepackage{amsfonts,amsmath,amssymb,amsthm,mathrsfs}
\usepackage{caption}
\usepackage[labelfont=rm]{subcaption}
\usepackage{xspace}
\usepackage{mathtools}
\usepackage{microtype}
\usepackage{enumerate}
\usepackage{booktabs}
\usepackage{hyperref}

\usepackage{tikz}
\usetikzlibrary{matrix,arrows,patterns,positioning,calc}
\usetikzlibrary{arrows,automata}
\usepackage{float}

\hypersetup{colorlinks,linkcolor=blue ,citecolor=blue, urlcolor=blue}

\DeclareSymbolFontAlphabet{\mathbb}{AMSb}
\newcommand\C{{\mathbb C}}

\newcommand\PP{{\mathbb P}}

\newcommand\R{{\mathbb R}}
\newcommand\Z{{\mathbb Z}}

\DeclareMathOperator{\conv}{conv}
\DeclareMathOperator{\rank}{rk}

\newcommand{\indeg}{\mathrm{indeg}}
\newcommand{\outdeg}{\mathrm{outdeg}}
\newcommand{\bzero}{\textnormal{\textbf{0}}}

\newcommand{\be}{\textnormal{\textbf{e}}}

\newcommand{\bq}{\textnormal{\textbf{q}}}

\newcommand{\bu}{\textnormal{\textbf{u}}}
\newcommand{\bv}{\textnormal{\textbf{v}}}

\newcommand{\bx}{\textnormal{\textbf{x}}}
\newcommand{\by}{\textnormal{\textbf{y}}}

\DeclareMathOperator{\Vol}{Vol}

\newcommand\arXiv[1]{\href{http://arxiv.org/abs/#1}{\texttt{arXiv:#1}}}

\theoremstyle{plain}
    \newtheorem{theorem}{Theorem}
    \newtheorem{corollary}[theorem]{Corollary}
    \newtheorem{lemma}[theorem]{Lemma}
    \newtheorem{proposition}[theorem]{Proposition}
\theoremstyle{definition}
    \newtheorem{remark}[theorem]{Remark}
    \newtheorem{example}[theorem]{Example}
    \newtheorem{definition}[theorem]{Definition}

    \newtheorem{question}{Question}

\title{Laplacian simplices associated to digraphs}

\author[G.\,Balletti]{Gabriele Balletti}
\address[G.\,Balletti]{Department of Mathematics\\Stockholm University\\SE-$106$\ $91$\
  Stockholm\\Sweden}
\email{balletti@math.su.se}

\author[T.\,Hibi]{Takayuki Hibi}
\address[T.\,Hibi]{Department of Pure and Applied Mathematics, Graduate School of Information Science and Technology, Osaka University, Suita, Osaka 565-0871, Japan}
\email{hibi@math.sci.osaka-u.ac.jp}

\author[M.\,Meyer]{Marie Meyer}
\address[M.\,Meyer]{Department of Mathematics\\University of Kentucky\\Lexington, KY 40506--0027}
\email{marie.meyer@uky.edu}

\author[A.\,Tsuchiya]{Akiyoshi Tsuchiya}
\address[A.\,Tsuchiya]{Department of Pure and Applied Mathematics, Graduate School of Information Science and Technology, Osaka University, Suita, Osaka 565-0871, Japan}
\email{a-tsuchiya@ist.osaka-u.ac.jp}


\begin{document}

\maketitle
\begin{abstract}
We associate to a finite digraph $D$ a lattice polytope $P_D$ whose vertices are the rows of the Laplacian matrix of $D$. This generalizes a construction introduced by Braun and the third author. As a consequence of the Matrix-Tree Theorem, we show that the normalized volume of $P_D$ equals the complexity of $D$, and $P_D$ contains the origin in its relative interior if and only if $D$ is strongly connected. Interesting connections with other families of simplices are established and then used to describe reflexivity, $h^*$-polynomial, and integer decomposition property of $P_D$ in these cases. We extend Braun and Meyer's study of cycles by considering cycle digraphs. In this setting we characterize reflexivity and show there are only four non-trivial reflexive Laplacian simplices having the integer decomposition property.

\end{abstract}

\section{Introduction}
The use of linear algebra to study properties of graphs is an established technique in combinatorics and consequently has led to the development of the so called \emph{spectral graph theory}. There is an extensive literature on algebraic aspects of spectral graph theory and on how combinatorial properties are encoded in characteristic polynomials, eigenvalues and eigenvectors of adjacency or Laplacian matrices of graphs (see \cite{Chu97} for a survey). It is tempting to take a step forward and associate a polytope $P_G$ to any graph $G$, by interpreting the rows of a matrix encoding the data of $G$ as the vertices of $P_G$. This is the case of the \emph{edge polytope} \cite{OH98,Vil98}, the convex hull of the rows of the unsigned vertex-edge incidence matrix of a graph, whose geometric and combinatorial properties have been extensively studied in the last two decades (see e.g. \cite{MHNH11,TZ14}), and used to build counterexamples \cite{OH99}. Recently the third author and Braun \cite{BM17} took a similar direction by associating to any graph $G$ the simplex $T_G$ (called the \emph{Laplacian simplex}) whose vertices are the rows of the Laplacian matrix of $G$. They established basic properties of $T_G$ and study reflexivity, the integer decomposition property, and unimodality of the Ehrhart $h^*$-vectors of $T_G$ for some special classes of graph.

Our contribution is to provide a more general setting for the investigation on Laplacian simplices. We do this by allowing $G$ to be a directed multigraph. In this way the objects studied in \cite{BM17} can be seen as a special case of our setting (see Remark~\ref{rmk:BM}). We have reasons to believe this generalization is the correct direction to take. Indeed, in the undirected and simple case, the origin of a Laplacian simplex coincides with the barycenter of its vertices, which is an uncommon property for a lattice simplex. In our setting, it is clarified (Proposition~\ref{prop:dependence} and Corollary~\ref{cor:0ininterior}) that this happens only for special digraphs, i.e. they need to be strongly connected and have the same number of spanning trees converging to each vertex. Moreover, in the original settings the volume of a Laplacian simplex associated to a graph with $n$ vertices equals $n$ times the number of spanning trees of the graph $G$. Extending to digraphs, it turns out (Proposition~\ref{prop:volume}) that the factor $n$ appears because in the undirected case each vertex has the same number of spanning trees converging to it.

\subsection*{Main results and organization of the paper}
In Section~\ref{sec:2} we set notation, basic definitions and prove the first important properties of Laplacian simplices in this new settings. In particular, we prove that the Laplacian simplex $P_D$ associated to a digraph $D$ with $n$ vertices, satisfies the following properties.
\begin{enumerate}
\item $P_D$ is a $(n-1)$-simplex if and only if $D$ has positive complexity, i.e. if $D$ has at least a spanning converging tree (Proposition~\ref{prop:characterization_rank}).
\end{enumerate}
Now assume that $D$ has positive complexity.
\begin{enumerate}
\setcounter{enumi}{1}
\item The numbers of spanning trees converging to each of the vertices of $D$ encode the barycentric coordinates of the origin with respect to the vertices of $P_D$ (Proposition~\ref{prop:dependence}).
\item $P_D$ contains the origin $\bzero$, which is in the strict relative interior of $P_D$ if and only if $D$ is strongly connected (Corollary~\ref{cor:0ininterior}).
\item The normalized volume of $P_D$ equals the total complexity of $D$, i.e. the total number of spanning converging trees (Proposition~\ref{prop:volume}).
\end{enumerate}

In Section~\ref{sec:3}, moreover we prove that, under some assumptions on $D$, $P_D$ is equivalent to the simplex associated to a weighted projective space (Proposition~\ref{prop:wps}). Under even more restrictive assumptions, $P_D$ is equivalent to one of the $\Delta_{(1,\bq)}$ simplices described in \cite{BDS16}. In such cases, we use this equivalence to characterize reflexivity (Corollary~\ref{cor:refl}), describe the Ehrhart $h^*$-polynomial, and the integer decomposition property of $P_D$ in terms of the spanning converging trees of $D$.

In Section~\ref{sec:4} we use these descriptions to extend the study of cycle graphs of Braun--Meyer \cite{BM17} to \emph{cycle digraphs}, i.e. strongly connected simple digraphs whose underlying graph is a cycle (see Definition~\ref{def:cycle}). We prove the following results.
\begin{enumerate}
\item In Proposition~\ref{prop:simple} we prove that a Laplacian simplex associated to a simple digraph has at most one interior lattice point. In particular, Laplacian simplices associated to cycle digraphs have exactly one interior point.
\item In Theorem~\ref{thm:terminal} we prove that the Laplacian simplices $P_D$ associated to a cycle digraph is terminal Fano, i.e. it contains no lattice points other than the origin and its vertices, unless $D$ is one of the following six digraphs.
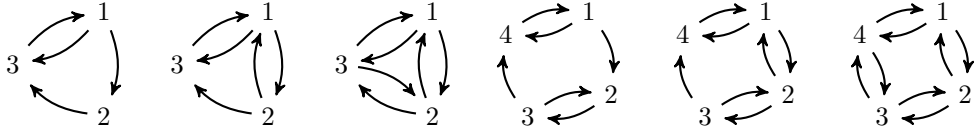
\begin{figure}[H]
\centering
\begin{tikzpicture}[->,>=stealth',shorten >=1pt,auto,node distance=2cm,
  thick,main node/.style={},scale=.8]

  \node[main node] at (60:1)	(1) {$1$};
  \node[main node] at (300:1)	(2) {$2$};
  \node[main node] at (180:1)	(3) {$3$};

  \path[every node/.style={font=\sffamily\small}]
    (1) edge [bend left=20] (2)
		edge [bend left=20] (3)
    (2) edge [bend left=20] (3)
    (3) edge [bend left=20] (1);
\end{tikzpicture}
\hspace*{.3cm}
\begin{tikzpicture}[->,>=stealth',shorten >=1pt,auto,node distance=2cm,
  thick,main node/.style={},scale=.8]

  \node[main node] at (60:1)	(1) {$1$};
  \node[main node] at (300:1)	(2) {$2$};
  \node[main node] at (180:1)	(3) {$3$};

  \path[every node/.style={font=\sffamily\small}]
    (1) edge [bend left=20] (2)
		edge [bend left=20] (3)
    (2) edge [bend left=20] (3)
		edge [bend left=20] (1)
    (3) edge [bend left=20] (1);
\end{tikzpicture}
\hspace*{.3cm}
\begin{tikzpicture}[->,>=stealth',shorten >=1pt,auto,node distance=2cm,
  thick,main node/.style={},scale=.8]

  \node[main node] at (60:1)	(1) {$1$};
  \node[main node] at (300:1)	(2) {$2$};
  \node[main node] at (180:1)	(3) {$3$};

  \path[every node/.style={font=\sffamily\small}]
    (1) edge [bend left=20] (2)
		edge [bend left=20] (3)
    (2) edge [bend left=20] (3)
		edge [bend left=20] (1)
    (3) edge [bend left=20] (1)
 		edge [bend left=20] (2);
\end{tikzpicture}
\hspace*{.3cm}
\begin{tikzpicture}[->,>=stealth',shorten >=1pt,auto,node distance=2cm,
  thick,main node/.style={},scale=.8]

  \node[main node] at (60:1)	(1) {$1$};
  \node[main node] at (330:1)	(2) {$2$};
  \node[main node] at (240:1)	(3) {$3$};
  \node[main node] at (150:1)	(4) {$4$};

  \path[every node/.style={font=\sffamily\small}]
    (1) edge [bend left=20] (2)
		edge [bend left=20] (4)
    (2) edge [bend left=20] (3)
    (3) edge [bend left=20] (4)
		edge [bend left=20] (2)
    (4) edge [bend left=20] (1);
\end{tikzpicture}
\hspace*{.3cm}
\begin{tikzpicture}[->,>=stealth',shorten >=1pt,auto,node distance=2cm,
  thick,main node/.style={},scale=.8]

  \node[main node] at (60:1)	(1) {$1$};
  \node[main node] at (330:1)	(2) {$2$};
  \node[main node] at (240:1)	(3) {$3$};
  \node[main node] at (150:1)	(4) {$4$};

  \path[every node/.style={font=\sffamily\small}]
    (1) edge [bend left=20] (2)
		edge [bend left=20] (4)
    (2) edge [bend left=20] (3)
		edge [bend left=20] (1)
    (3) edge [bend left=20] (4)
		edge [bend left=20] (2)
    (4) edge [bend left=20] (1);
\end{tikzpicture}
\hspace*{.3cm}
\begin{tikzpicture}[->,>=stealth',shorten >=1pt,auto,node distance=2cm,
  thick,main node/.style={},scale=.8]

  \node[main node] at (60:1)	(1) {$1$};
  \node[main node] at (330:1)	(2) {$2$};
  \node[main node] at (240:1)	(3) {$3$};
  \node[main node] at (150:1)	(4) {$4$};

  \path[every node/.style={font=\sffamily\small}]
    (1) edge [bend left=20] (2)
		edge [bend left=20] (4)
    (2) edge [bend left=20] (3)
		edge [bend left=20] (1)
    (3) edge [bend left=20] (4)
		edge [bend left=20] (2)
    (4) edge [bend left=20] (1)
		edge [bend left=20] (3);
\end{tikzpicture}
\end{figure}
\item In Theorem~\ref{thm:ref_dir_cycles} we characterize the reflexivity of $P_D$ in terms of combinatorial properties of the cycle digraph $D$.
\item In Theorem~\ref{thm_IDP} we prove that a reflexive Laplacian simplex $P_D$ has the integer decomposition property if and only if $D$ is the oriented cycle $1 \to 2 \to \cdots \to n-1 \to n \to 1$ for any $n$, or one of the following four exceptional digraphs.
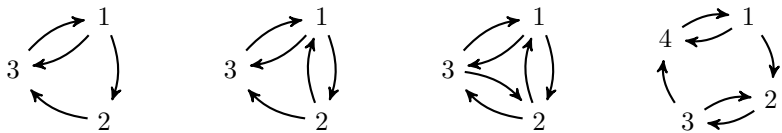
\begin{figure}[H]
\centering
\begin{tikzpicture}[->,>=stealth',shorten >=1pt,auto,node distance=2cm,
  thick,main node/.style={},scale=.8]

  \node[main node] at (60:1)	(1) {$1$};
  \node[main node] at (300:1)	(2) {$2$};
  \node[main node] at (180:1)	(3) {$3$};

  \path[every node/.style={font=\sffamily\small}]
    (1) edge [bend left=20] (2)
		edge [bend left=20] (3)
    (2) edge [bend left=20] (3)
    (3) edge [bend left=20] (1);
\end{tikzpicture}
\hspace*{1cm}
\begin{tikzpicture}[->,>=stealth',shorten >=1pt,auto,node distance=2cm,
  thick,main node/.style={},scale=.8]

  \node[main node] at (60:1)	(1) {$1$};
  \node[main node] at (300:1)	(2) {$2$};
  \node[main node] at (180:1)	(3) {$3$};

  \path[every node/.style={font=\sffamily\small}]
    (1) edge [bend left=20] (2)
		edge [bend left=20] (3)
    (2) edge [bend left=20] (3)
		edge [bend left=20] (1)
    (3) edge [bend left=20] (1);
\end{tikzpicture}
\hspace*{1cm}
\begin{tikzpicture}[->,>=stealth',shorten >=1pt,auto,node distance=2cm,
  thick,main node/.style={},scale=.8]

  \node[main node] at (60:1)	(1) {$1$};
  \node[main node] at (300:1)	(2) {$2$};
  \node[main node] at (180:1)	(3) {$3$};

  \path[every node/.style={font=\sffamily\small}]
    (1) edge [bend left=20] (2)
		edge [bend left=20] (3)
    (2) edge [bend left=20] (3)
		edge [bend left=20] (1)
    (3) edge [bend left=20] (1)
 		edge [bend left=20] (2);
\end{tikzpicture}
\hspace*{1cm}
\begin{tikzpicture}[->,>=stealth',shorten >=1pt,auto,node distance=2cm,
  thick,main node/.style={},scale=.8]

  \node[main node] at (60:1)	(1) {$1$};
  \node[main node] at (330:1)	(2) {$2$};
  \node[main node] at (240:1)	(3) {$3$};
  \node[main node] at (150:1)	(4) {$4$};

  \path[every node/.style={font=\sffamily\small}]
    (1) edge [bend left=20] (2)
		edge [bend left=20] (4)
    (2) edge [bend left=20] (3)
    (3) edge [bend left=20] (4)
		edge [bend left=20] (2)
    (4) edge [bend left=20] (1);
\end{tikzpicture}
\end{figure}
\item In Theorem~\ref{thm:non_unim} we construct a family of reflexive Laplacian simplices with symmetric but non-unimodal $h^*$-vector $(1,\ldots,1,2,\ldots,2,1,\ldots,1,2,\ldots,2,1,\ldots,1,2,\ldots,2,1,\ldots,1)$.
\end{enumerate}

In the final Section~\ref{sec:underlying} we try to understand how the structure of the underlying simple and undirected graph of a digraph affects the reflexivity of its Laplacian polytope. In particular we show that there is a graph which is not the underlying graph of any simple directed graph whose Laplacian simplex is reflexive. A more general version of this problem (Question~\ref{q3}) remains open.

\section*{Acknowledgments}
The authors would like to thank Benjamin Braun for helpful comments.
The project started during the ``Workshop on Convex Polytopes for Graduate Students'' at Osaka University. The first author is partially supported by the Vetenskapsr{\aa}det grant~NT:2014-3991. The fourth author is partially supported by Grant-in-Aid for JSPS Fellows 16J01549.
\section{The Laplacian Polytope Construction and properties of Laplacian simplices}
\label{sec:2}

\subsection{The Laplacian of a digraph}
Let $D$ be a finite directed graph (\emph{digraph}) on the vertex set $V(D)=[n]$, where $[n] \coloneqq \{ 1 , \ldots , n \}$. Let $E(D)$ be the set of the \emph{directed edges} of $D$. A directed edge $e = (i,j) \in E(D)$ points from a vertex $i$ (called the \emph{tail} of $e$) to another vertex $j$ (called the \emph{head} of $e$). Multiple directed edges between vertices are allowed, and we denote by $a_{i,j}$ the number of directed edges having tail on the vertex $i$ and head on the vertex $j$ of $D$, with $i,j\in [n]$ and $i \neq j$. Since loops do not affect the Laplacian matrix, we assume $D$ to be without loops, and thus $a_{i,i}=0$ for all $i \in [n]$. The number of edges with vertex $i$ as a tail is called the \emph{outdegree} of $i$ and is denoted by $\outdeg(i)$, while the number of edges with vertex $i$ as a head is called the \emph{indegree} of $i$ and is denoted by $\indeg(i)$. We call $D$ \emph{strongly connected} if it contains a directed path from $i$ to $j$ for every pair of vertices $i,j \in [n]$ and \emph{weakly connected} if there exists a path (not necessarily directed) between $i$ and $j$ for every pair of vertices $i, j \in [n]$. 
In this paper we assume $D$ has no isolated vertices, i.e. vertices with indegree and outdegree equal to zero. A \emph{converging tree} is a weakly connected digraph having one vertex with outdegree zero, called the \emph{root} of the tree, while all other vertices have outdegree one.
%A \emph{converging forest} is a digraph having converging trees as weakly connected components. The \emph{set of roots} of a converging forest are the roots of the converging trees in the forest.
We say that a subgraph $D'$ of $D$ is \emph{spanning} if the vertex set of $D'$ is $[n]$.
%We denote by $d_D$ the minimum number of disjoint converging trees a spanning converging forest of $D$ can have.

All the data of $D$ can be encoded in the \emph{adjacency matrix} of $D$, that is, the $n \times n$ matrix $A(D) \coloneqq (a_{i,j})_{1 \leq i,j \leq n}$. We define the \emph{outdegree matrix} of $D$ to be $O(D) \coloneqq (d_{i,j})_{1 \leq i,j \leq n}$, the $n \times n$ matrix with $d_{i,j}=\outdeg(i)$, if $i=j$, and $d_{i,j}=0$ otherwise. We define the \emph{Laplacian matrix} of $D$ to be the matrix $L(D) \coloneqq O(D) - A(D)$.

Observe the sum of the entries of each row of $L(D)$ is zero. Thus the rank of the Laplacian matrix is never maximal, i.e.
\begin{equation}
\label{eq:rank}
\rank (L(D)) \leq n-1.
\end{equation}
A combinatorial interpretation for having equality in \eqref{eq:rank} is given by the Matrix-Tree Theorem, which is presented here in its generalized version for digraphs. The interpretation is given in terms of spanning converging trees of $D$. For any $i \in [n]$, we denote by $c_i$ the number of spanning trees which converge to $i$, i.e. the converging trees of $D$ with $n$ vertices having $i$ as the root. We denote by $c(D)$ the total number of converging spanning trees of $D$, i.e. $c(D) \coloneqq \sum_{i = 1}^n c_i$. The number $c(D)$ is usually referred to as the \emph{complexity} of the digraph $D$.

\begin{theorem}[{Matrix-Tree Theorem \cite[Theorem~5.6.4]{Sta99}}]
\label{thm:MTT}
Let $D$ be a digraph without loops on the vertex set $[n]$. Let $i,j \in [n]$, and $L(D)_{i,j}$ the matrix obtained from $L(D)$ by removing its $i$-th row and $j$-th column. Then the determinant of $L(D)_{i,j}$ equals, up to a change of sign, the number of spanning trees of $D$ converging to $i$, i.e.
\[
(-1)^{i+j} \det L(D)_{i,j} = \det L(D)_{i,i} = c_i.
\]
In particular the complexity of $D$ is
\[
c(D)=\sum_{i=1}^n \det L(D)_{i,i}.
\]
\end{theorem}

%Note that $0$ is both a left and right eigenvalue of $L(D)$, while $\bone=(1,\ldots,1)$ is always a right eigenvector associated to $0$. By Gershgorin circle theorem, all the nonzero eigenvalues of $L(D)$ have positive real part. The sum of the eigenvalues equals the trace of $L(D)$, and therefore the total weight of $D$, i.e. the sum of the weights of all the arcs.

\subsection{The Laplacian polytope associated to a digraph}
Let $D$ be a digraph on the vertex set $[n]$. To $D$ we associate a convex polytope in $\R^n$ having vertices in the integer lattice $\Z^n$. We call the \emph{Laplacian polytope} associated to $D$ the polytope $P_D:=\conv({\{\bv_1, \ldots, \bv_n\}}) \subseteq \R^n$, where $\bv_i$ is the $i$-th row of the Laplacian matrix of $D$. The polytope $P_D$ is not full dimensional; since the sum of the entries in each row of $L(D)$ vanishes, $P_D$ is contained in the hyperplane $H \coloneqq \{\bx=(x_1,\ldots,x_n) : \sum_{i=1}^n x_i = 0 \}$ of $\R^n$. In particular, the dimension of the Laplacian polytope, $\dim(P_D)$, equals the rank of the Laplacian matrix $L(D)$. When the rank of $L(D)$ is equal to $n-1$, then $P_D$ is a simplex, called the \emph{Laplacian simplex} associated to $D$. 

\begin{remark}
\label{rmk:BM}
The Laplacian simplex in this context is a generalization of the Laplacian simplex introduced by Braun-Meyer in \cite{BM17}.
For a connected simple graph $G$, they define the simplex $T_G$ as the convex hull of the rows of the graph Laplacian matrix of $G$.
The Laplacian $L(G)$ of $G$ can be interpreted as the Laplacian of a digraph $D_G$, and thus the resulting simplices are equal, that is, $T_G = P_{D_G}$.  
\end{remark}

Two lattice polytopes $P,Q\subset \R^n$ are said to be \emph{unimodularly equivalent} if there exists an affine lattice automorphism $\varphi\in \textrm{GL}_n(\Z)\ltimes\Z^n$ of $\Z^n$ such that $\varphi_\R(P)=Q$. Sometimes it is convenient to work with full dimensional lattice polytopes, i.e. lattice polytopes embedded in a space of their same dimension. Given a Laplacian simplex $P_D$, one can easily get a full dimensional unimodularly equivalent copy of $P_D$ by considering the lattice polytope defined as the convex hull of the rows of $L(D)$ with one column deleted. An example of this can be observed in Example~\ref{ex:first_example}.

\begin{example}
\label{ex:first_example}
Let $D$ be the following digraph with its Laplacian matrix $L(D)$.

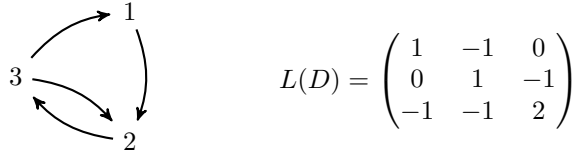
\begin{center}
\begin{tabular}{c}
\begin{tikzpicture}[->,>=stealth',shorten >=1pt,auto,node distance=2cm,
  thick,main node/.style={},scale=1]

  \node[main node] at (60:1)	(1) {$1$};
  \node[main node] at (300:1)	(2) {$2$};
  \node[main node] at (180:1)	(3) {$3$};

  \path[every node/.style={font=\sffamily\small}]
    (1) edge [bend left=20] (2)
    (2) edge [bend left=20] (3)
    (3) edge [bend left=20] (1)
	(3) edge [bend left=20] (2);
\end{tikzpicture}
\end{tabular}
  \qquad \qquad 
  $L(D)=
\begin{pmatrix}
1 & -1 & 0 \\
0 & 1  & -1\\
-1& -1 & 2
\end{pmatrix}$
\end{center}
Note that $L(D)$ has rank two, which means $P_D$ is a two dimensional simplex in $\R^3$. Full dimensional unimodularly equivalent copies of $P_D$ can be obtained by deleting any of the columns of $L(D)$ and considering the convex hull of the rows as in Figure~\ref{fig:ex1}.

\begin{center}
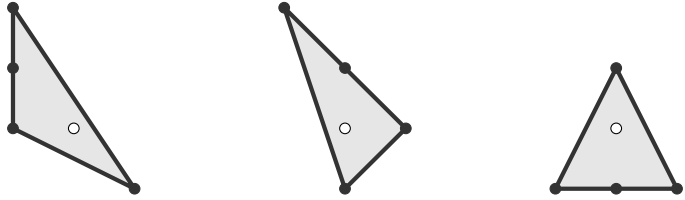
\begin{figure}[h]
\begin{tikzpicture}[scale=.8]
\fill[black!10] (-1,0)--(1,-1)--(-1,2)--cycle;
\draw[ultra thick,black!80,line join=round] (-1,0)--(1,-1)--(-1,2)--cycle;
\draw[fill=black!80,black!80] (-1,0) circle (0.25em) node[below]{};
\draw[fill=black!80,black!80] (-1,1) circle (0.25em) node[below]{};
\draw[fill=black!80,black!80] (-1,2) circle (0.25em) node[below]{};
\draw[black,fill=white] (0,0) circle (0.25em) node[below]{};
\draw[fill=black!80,black!80] (1,-1) circle (0.25em) node[below]{};
\end{tikzpicture}
\hspace*{1.5cm}
\begin{tikzpicture}[scale=.8]
\fill[black!10] (1,0)--(0,-1)--(-1,2)--cycle;
\draw[ultra thick,black!80,line join=round] (1,0)--(0,-1)--(-1,2)--cycle;
\draw[fill=black!80,black!80] (-1,2) circle (0.25em) node[below]{};
\draw[fill=black!80,black!80] (0,-1) circle (0.25em) node[below]{};
\draw[fill=black!80,black!80] (1,0) circle (0.25em) node[below]{};
\draw[black,fill=white] (0,0) circle (0.25em) node[below]{};
\draw[fill=black!80,black!80] (0,1) circle (0.25em) node[below]{};
\end{tikzpicture}
\hspace*{1.5cm}
\begin{tikzpicture}[scale=.8]
\fill[black!10] (1,-1)--(0,1)--(-1,-1)--cycle;
\draw[ultra thick,black!80,line join=round] (1,-1)--(0,1)--(-1,-1)--cycle;
\draw[fill=black!80,black!80] (-1,-1) circle (0.25em) node[below]{};
\draw[fill=black!80,black!80] (0,-1) circle (0.25em) node[below]{};
\draw[fill=black!80,black!80] (1,-1) circle (0.25em) node[below]{};
\draw[black,fill=white] (0,0) circle (0.25em) node[below]{};
\draw[fill=black!80,black!80] (0,1) circle (0.25em) node[below]{};
\end{tikzpicture}
\caption{\label{fig:ex1} Three unimodularly equivalent full dimensional copies of $P_D$ obtained by deleting the first, second, and third columns of $L(D)$, respectively.}
\end{figure} 
\end{center}
\end{example}

\subsection{Properties of Laplacian simplices}
From the Matrix-Tree Theorem (Theorem~\ref{thm:MTT}) the following characterization can be immediately obtained.

\begin{proposition}
\label{prop:characterization_rank}
Let $D$ be a digraph on $n$ vertices. The following are equivalent:
\begin{enumerate}[\qquad 1.]
\item $D$ has positive complexity $c(D)$;
\item $\rank ( L(D) ) = n-1$;
\item $P_D$ is an $(n-1)$-simplex.
\end{enumerate}
\end{proposition}

Following the work of Braun-Meyer \cite{BM17}, we focus our attention to the case in which a digraph $D$ on $n$ vertices defines an $(n-1)$-simplex. Proposition~\ref{prop:characterization_rank} asserts we will always assume the digraph $D$ has positive complexity. As another consequence of Theorem~\ref{thm:MTT}, we deduce the following result.

\begin{proposition}
\label{prop:dependence}
Let $D$ a digraph with positive complexity. Then the numbers of spanning converging trees $c_1,\ldots,c_n$ of $D$ encode the unique linear dependence among the vertices $\bv_1, \ldots, \bv_n$ of $P_D$, i.e.
\[
\sum_{i=1}^n c_i \bv_i = \bzero.
\]
\end{proposition}
\begin{proof}
Since the determinant of $L(D)$ is zero, the Laplace expansion along the $j$-th column of $L(D)$ yields $\sum_{i=1}^n (-1)^{i+j}\det{L(D)_{i,j}}v_{i,j} = 0$, where $L(D)_{i,j}$ is the matrix of $L(D)$ obtained by removing the $i$-th row and $j$-th column of $L(D)$, and $v_{i,j}$ is the $j$-th entry of $\bv_i$. By Theorem~\ref{thm:MTT}, $\det{L(D)_{i,i}} = c_i$.
\end{proof}

\begin{corollary}
\label{cor:0ininterior}
Let $D$ be a digraph on $n$ vertices having positive complexity. Then $\bzero \in P_D$. Moreover $\bzero$ is an interior point of $P_D$ if and only if $D$ is strongly connected.
\end{corollary}
\begin{proof}
The first statement is a direct consequence of Proposition~\ref{prop:dependence}. For the second it is enough to note that $D$ is strongly connected if and only if each vertex has at least one spanning converging tree.
\end{proof}

In this setting we prove a formula for the normalized volume of $P_D$. If a polytope $P$ is $n$-dimensional, its \emph{normalized volume} $\Vol(P)$ is defined to be $n!$ times the Euclidean volume of $P$.

\begin{proposition}
\label{prop:volume}
Let $D$ be a digraph with positive complexity. Then its normalized volume equals the complexity of $D$, i.e.
\[
\Vol(P_D) = c(D),
\]
\end{proposition}
\begin{proof}
In this case $P_D$ is a $(n-1)$-simplex by Proposition~\ref{prop:characterization_rank}. For $i=1,\ldots,n$, we denote by $F_i$ the facet of $P_D$ not containing the vertex $i$. By Proposition~\ref{prop:dependence}, $\bzero \in P_D$, so we can triangulate $P_D$ with the $S_i \coloneqq \conv(\{\bzero\} \cup  F_i)$. In particular
\[
\Vol(P_D) = \sum_{i=1}^n \Vol(S_i),
\]
where we consider $\Vol(S_i)=0$ if $S_i$ is not of dimension $n-1$. Let $S'_i$ be the unimodular copy of $S_i$ obtained as the convex hull of the rows of $L(D)_{i,i}$, the matrix obtained from $L(D)$ by removing the $i$-th row and $i$-th column.
\[
\Vol(P_D) = \sum_{i=1}^n \Vol(S_i) = \sum_{i=1}^n \Vol(S'_i) = \sum_{i=1}^n \det L(D)_{i,i} = \sum_{i=1}^n c_i,
\]
where the last equality follows from Theorem~\ref{thm:MTT}.
\end{proof}

\section{Connections with other families of simplices}
\label{sec:3}
Laplacian simplices associated to strongly connected digraphs have interesting intersections with the study of weighted projective space arising from algebraic geometry and with the study of other families of simplices. We use these connections to describe properties of Laplacian simplices with particular attention to reflexivity, the integer decomposition property, and $h^*$-vectors of lattice polytopes. For the convenience of the reader, the next two subsections are a quick introduction to these topics.

\subsection{Weighted projective spaces}
Given positive integers $\lambda_1,\ldots,\lambda_n$ which are coprime, i.e. such that $\gcd \{ \lambda_1 , \ldots , \lambda_n \} = 1$, we define the polynomial algebra $S(\lambda_1 , \ldots , \lambda_n) \coloneqq \C[x_1,\ldots,x_n]$ graded by $\deg x_i \coloneqq \lambda_i$. A \emph{weighted projective space} with weights $\lambda_1 , \ldots , \lambda_n$ is the projective variety $\PP(\lambda_1 , \ldots , \lambda_n) \coloneqq \textrm{Proj}(S(\lambda_1 , \ldots , \lambda_n))$. Since $\PP(\lambda_1 , \ldots , \lambda_n)$ is a toric variety, it corresponds to a fan $\Delta$ that can be characterized as follows. Let $\bv_1,\ldots,\bv_n$ be primitive lattice points which generate the lattice and satisfy $\sum_{i=1}^n \lambda_i \bv_i = \bzero$, where $\gcd \{ \lambda_1 , \ldots , \lambda_n \} = 1$. Then, up to isomorphism, the fan $\Delta$ is the fan whose rays are generated by the $\bv_i$. Note that the fan $\Delta$ identifies uniquely the simplex $S_\Delta \coloneqq \conv (\{ \bv_1, \ldots, \bv_n\})$. With an abuse of terminology, we say that a simplex is the weighted projective space $\PP(\lambda_1 , \ldots , \lambda_n)$ if it is unimodularly equivalent to the simplex $S_\Delta$. For details we refer the reader to \cite{Dol82,Ian00}. 

\subsection{Ehrhart Theory, reflexivity and integer decomposition properties of lattice polytopes}
\label{ssec:ehrhart_theory}
For a proper introduction to Ehrhart Theory and related topics, we refer to the textbook \cite{BR15}. A classical result by Ehrhart states that the number of lattice points in integer dilations of a $d$-dimensional lattice polytope $P\subset \R^d$ behaves polynomially. In terms of generating series this translates into the equality
\[
1 + \sum_{n\ge1} | nP \cap \Z^d | z^n = \frac{h_d^*z^d + \cdots + h_1^*z + 1}{(1-z)^{d+1}},
\]
where $h^*(z) = h_d^*z^d + \cdots + h_1^*z + 1$ is a polynomial with non-negative integer coefficients, called the \emph{$h^*$-polynomial} of $P$. This is an important invariant as it preserves much information about $P$. For example, the following relations are well known:
\[
h^*_1=|P \cap \Z^d|-d-1, \qquad h^*_d=|P^\circ \cap \Z^d|, \qquad 1 + \sum_{i=1}^d h^*_i=\Vol(P),
\]
where $P^\circ$ denotes the relative interior of $P$. The $h^*$-polynomial of $P$ is often identified with the vector of its coefficient $(1,h^*_1,\ldots,h^*_d)$, called the \emph{$h^*$-vector} of $P$. We call a vector $(x_0,x_1,\ldots,x_d)$ \emph{unimodal} if there exists a $1 \leq j \leq d$ such that $x_i \leq x_{i+1}$ for all $0 \leq i < j$ and $x_k \geq x_{k+1}$ for all $j \leq k < d$. An important open problem in Ehrhart Theory is to understand under which conditions $h^*$-vectors are unimodal (see \cite{Bra16} for a survey).

The \emph{dual polytope} $P^*$ of a full dimensional rational polytope $P$ which contains the origin is the polytope $P^* \coloneqq \{ \bx \in \R^d \mid \bx \cdot \by \le 1 \mbox{ for all } \by \in P \}$. If $P$ is a lattice polytope, we call it \emph{reflexive} if its dual $P^*$ is again a lattice polytope. We extend the definition of reflexive to all the lattice polytopes which are unimodular equivalent to $P$. Reflexive polytopes were first introduced in \cite{Bat94}. A well known result of the second author \cite{Hib92} characterizes reflexive polytopes as lattice polytopes having a \emph{symmetric} $h^*$-vector, i.e. such that $h^*_i = h^*_{d-i}$ for $0 \leq i \leq \lfloor \frac{d}{2} \rfloor$.

We say that a lattice polytope $P$ has the \emph{integer decomposition property}, if, for every positive integer $n$ and for all $\bx \in nP \cap \Z^d$ there exist $\bx_1,\ldots,\bx_n \in P \cap \Z^d$ such that $\bx = \bx_1 + \ldots + \bx_n$. A polytope having the integer decomposition property is often called \emph{IDP}.

Many efforts have been made to find sufficient conditions for unimodality. It has been conjectured by Stanley \cite{MR1110850} that a standard graded Cohen-Macaulay integral domain has unimodal $h$-vector. Although this has proven to be wrong in general \cite{MR1310575}, in the context of lattice polytope this can be translated in the following question which remains open: does an IDP polytope always have unimodal $h^*$-vector? A weaker statement of this question has also been suggested by Ohsugi and the second author \cite{OH06}, who conjectured that being reflexive and IDP is a sufficient condition for a lattice polytope to have unimodular $h^*$-vector.

\subsection{Connections with weighted projective spaces and $\Delta_{(1,\bq)}$-simplices}
We now relate Laplacian polytopes to weighted projective spaces.
\begin{proposition}
\label{prop:wps}
Let $D$ be a strongly connected digraph such that $\gcd\{c_1 , \ldots, c_n \}=1$. Then $P_D$ is equivalent to the weighted projective space $\PP(c_1 , \ldots, c_n)$.
\end{proposition}
\begin{proof}
By Corollary~\ref{cor:0ininterior}, $\sum_{i=1}^n \lambda_i \bv_i = \bzero$; so we just need to prove that the vertices of $P_D$ span the lattice. Let $L$ be the lattice spanned by all the vertices, and $L_i$ the lattice spanned by all the vertices $\bv_j$ such that $j \neq i$. Then we have the following inclusions of subgroups of $\Z^d$: $L_i \subset L \subset \Z^n$. In particular for all $i$, $|\Z^n : L| |L : L_i|=|\Z^n : L_i|=c_i$, which implies that $ L = \Z^n$.
\end{proof}

In \cite{Con02,Kas13} characterizations for properties of weighted projective spaces are given in terms of their weights and are used to perform classifications. We use these results to translate properties of $D$ to properties of $P_D$. Motivated by the open questions mentioned in the previous subsection, we focus on reflexivity, integer decomposition property, and description of the $h^*$-polynomial.  

We use the following result of Conrads, presented below in a slightly weaker form.

\begin{proposition}[{\cite[Proposition~5.1]{Con02}}]
\label{prop:conrads}
Let $S  = \conv (\bv_1,\ldots,\bv_n)$ be an $(n-1)$-simplex such that $\sum_{i=1}^n q_i \bv_i = \bzero$ for some positive integers $q_1,\ldots,q_n$ satisfying $\gcd(q_1,\ldots,q_n)=1$. Then $S$ is reflexive if and only if
\begin{equation}
\label{eq:conrads}
q_i \quad \text{divides the total weight} \quad |Q|=\sum\limits_{j=1}^{n} q_j \quad for \quad i=1,\ldots,n.
\end{equation}
%Conversely, if $q_1,\ldots,q_n \in \N$ are given such that \eqref{eq:conrads} holds, then the weighted projective space $\mathbb{P}(q_1,\ldots,q_n)$ has at most Gorenstein singularities. 
\end{proposition}

From this we can derive the following corollary.

\begin{corollary}
\label{cor:refl}
Let $D$ be a strongly connected digraph such that $\gcd\{c_1 , \ldots, c_n \}=1$. Then $P_D$ is reflexive if and only if $c_i | c(D)$ for all $i$.
\end{corollary}

Proposition~\ref{prop:conrads} is also used by Braun--Davis--Solus \cite{BDS16} to define an interesting class of reflexive simplices. In particular they are interested in studying the integer decomposition property and unimodality of the $h^*$-vectors of such simplices constructed the following way.
Let $\bq=(q_1,\ldots,q_n)$ be an nondecreasing sequence of positive integers satisfying the condition $q_j | (1+\sum_{i \neq j}q_i )$
for all $j=1,\ldots,n$. For such a vector $\bq$, the simplex $\Delta_{(1,\bq)}$ is defined as
\[
	\Delta_{(1,\bq)}:=\conv \left\{ \be_1,\be_2,\ldots,\be_n,-\sum_{i=1}^n q_i \be_i \right\}, 
\]
where $e_i\in \R^n$ is the $i$-th standard basis vector. By Proposition~\ref{prop:conrads}, $\Delta_{(1,\bq)}$ is a reflexive simplex. Note that $\Delta_{(1,\bq)}$ is equivalent to the weighted projective space with weights $(1,q_1,\ldots,q_n)$.

The next proposition shows that the simplices $\Delta_{(1,\bq)}$ are a subfamily of Laplacian simplices arising from special star-shaped strongly connected digraphs.  

\begin{proposition}
\label{prop:Delta_1q}
Let $\bq=(q_1,q_2,\ldots,q_n)$ be any nondecreasing sequence of positive integers. Then there is a strongly connected digraph $D$ such that $P_D \cong \PP(1,q_1,\ldots,q_n)$.
In particular,  if $\bq$ satisfies the condition $q_j | (1+\sum_{i \neq j}q_i )$
for all $j=1,\ldots,n$, then $P_D \cong \Delta_{(1,\bq)}$.
\end{proposition}
\begin{proof}
As in Figure~\ref{fig:star_shaped}, we define $D$ as the star-shaped digraph on the vertices $1,\ldots,n+1$ such that
\begin{enumerate}
\item for $i=1,\ldots,n$ there are $q_i$ many edges directed from $1$ to $i+1$;
\item for $i=1,\ldots,n$ there is one edge directed from $i+1$ to $1$.
\end{enumerate} 
It is easy to verify that $c_1=1$ and, for $i \geq 2$, $c_i = q_i$. Proposition~\ref{prop:wps} concludes the proof.
\end{proof}

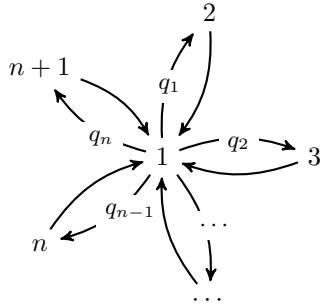
\begin{figure}[h]
\begin{tikzpicture}[->,>=stealth',shorten >=1pt,auto,node distance=2cm,
  thick,main node/.style={}]

  \node[main node] at (0:0)	(1) {$1$};
  \node[main node] at (72:2)	(2) {$2$};
  \node[main node] at (0:2)	(3) {$3$};
  \node[main node] at (288:2)	(4) {$\cdots$};
  \node[main node] at (216:2)	(5) {$n$};
  \node[main node] at (144:2)	(6) {$n+1$};

  \path[every node/.style={font=\sffamily\small,fill=white}]
    (1) edge [bend left=20] node[xshift=.3cm,yshift=-.3cm] {$q_1$} (2)
		edge [bend left=20] node[xshift=.0cm,yshift=-.3cm] {$q_2$} (3)
		edge [bend left=20] node[xshift=-.2cm,yshift=-.2cm] {$\cdots$} (4)
		edge [bend left=20] node[xshift=-.2cm,yshift=.3cm] {$q_{n-1}$} (5)
		edge [bend left=20] node[xshift=.4cm,yshift=.1cm] {$q_{n}$} (6)
    (2) edge [bend left=20] (1)
    (3) edge [bend left=20] (1)
    (4) edge [bend left=20] (1)
    (5) edge [bend left=20] (1)
    (6) edge [bend left=20] (1);

\end{tikzpicture}
\caption{The star shaped digraph $D$ such that $P_D=\PP(1,q_1, \ldots , q_n)$. The label on an edge from $i$ to $j$ represent the total number of edges from $i$ to $j$.}
\label{fig:star_shaped}
\end{figure}

In \cite{BDS16} an explicit formula for the $h^*$-polynomial of the simplices $\Delta_{(1,\bq)}$ is given. We remark that such formula can be also extracted from \cite{Kas13}, where it is proved in the more general setting of weighted projective spaces; however, the formulation given in \cite{BDS16} perfectly fits our needs.

\begin{theorem}[{\cite[Theorem~2.5]{BDS16}}]
\label{thm:hstarq}
The $h^*$-polynomial of $\Delta_{(1,\bq)}$ is 
\[
h^*(\Delta_{(1,\bq)};z) = \sum_{b=0}^{q_1+\cdots +q_n}z^{w(b)}
\]
where 
\[
w(b):=b-\sum_{i=1}^n\left\lfloor\frac{q_ib}{1+q_1+\cdots +q_n} \right\rfloor \, .
\]
\end{theorem}

Finally, in \cite{BDS16}, necessary conditions for a $\Delta_{(1,\bq)}$ simplex to be IDP are given.

\begin{lemma}[{\cite[Corollary 2.7]{BDS16}}]
\label{lem:BDS_IDP}
If $\Delta_{(1,\bq)}$ is IDP, then for all $j=1,2,\ldots,n$
\[
\frac{1}{q_j}+\sum_{i\neq j}\left\{\frac{q_i}{q_j}\right\}=1.
\]
\end{lemma}

\section{Laplacian simplices associated to cycle digraphs}
\label{sec:4}
We now want to extend the study of Braun--Meyer on simplices associated to cycle graphs. They show that the Laplacian simplex associated to a cycle is reflexive if and only if the cycle has odd length $n$; in that case it has a unimodal $h^*$-vector and fails to be IDP for $n \ge 5$ \cite[Section~5]{BM17}. We generalize their study by extending the notion of cycle graphs to cycle digraphs. 
A natural way to do it, would be to consider digraphs whose underlying simple graphs are cycle graphs. Here by underlying simple graph $G_D$ of a digraph $D$ we mean the simple undirected graph on the vertex set $V(G_D) \coloneqq V(D)$ such that the edge $\{i,j\}$ is in $E(G_D)$ if and only if there is at least one directed edge between $i$ and $j$ in $D$ (in either of the two directions). But since we are interested in reflexivity, we know, by Corollary~\ref{cor:0ininterior}, that $D$ has to be strongly connected, therefore $D$ needs to contain a cycle entirely oriented in one of the two possible directions. This generalization of cycle graphs will be made clear later (Definition~\ref{def:cycle}). Moreover, in order to ensure the presence of no more than one interior point, we will assume that $D$ is a simple digraph. In the next subsection we show that this is always true for any simple graph.

\subsection{Laplacian simplices associated to simple digraphs} In this section we focus on \emph{simple} digraphs, where by simple we mean there is at most one directed edge from $i$ to $j$, for any pair of vertices $i,j \in [n]$, $i \neq j$. Note that the presence of both a directed edge from $i$ to $j$ and one from $j$ to $i$ is allowed.  As in the previous section, we restrict our attention to those digraphs having positive complexity. Note this case still generalizes the work of Braun--Meyer \cite{BM17} (see Remark~\ref{rmk:BM}) and defines polytopes with at most one interior point. Indeed we prove that all the Laplacian polytopes of a simple directed graph on $n$ vertices are subpolytopes of $P_{K_n}$, the Laplacian simplex associated to the complete simple digraph. Observe $P_{K_n}$ is equivalent to the $n$-th dilation of an $(n-1)$-dimensional unimodular simplex, and therefore it has exactly one interior lattice point.

\begin{proposition}
\label{prop:simple}
Let $D$ be a simple digraph on $n$ vertices. Then $P_D$ is a subpolytope of $P_{K_n}$. In particular, if $D$ is strongly connected, then $P_D$ has exactly one interior lattice point.
\end{proposition}
\begin{proof}
Corollary~\ref{cor:0ininterior} implies that $P_D$ has at least one interior lattice point, so the second statement follows directly from the first one. In order to prove the first part, we show that any vertex $\bu$ of $P_D$ is in $P_{K_n}$. Up to a relabeling of the vertices, we can assume that $\bu=(a,-1,\ldots,-1,0,\ldots,0)$, where $a$ equals the number entries of $u$ which are equal to $-1$. We know that the Laplacian $L(K_n)$ is
\[
L(K_n) =
\begin{bmatrix}
    n-1 & -1  & \dots  & -1 \\
    -1 & n-1  & \dots  & -1 \\
    \vdots & \vdots & \ddots & \vdots \\
    -1 & -1 & \dots  & n-1
\end{bmatrix}.
\]
We denote by $\bv_i$ the $i$-th row of $L(K_n)$, as well as the corresponding vertex of $P_{K_n}$. It is then enough to prove that $\bu$ can be written as a convex combination of the vertices of $K_n$, i.e. that $\bu = \sum_{i=0}^n \lambda_i \bv _i$, with $0 \leq \lambda_i \leq 1$ and $\sum_{i=0}^n \lambda_i = 1$. This can be done with the following choice of barycentric coordinates:
\[
\lambda_i =
\begin{cases}
\frac{a+1}{n}, & \mbox{if } i = 1 \\
0, & \mbox{if } 2 \leq i \leq a+1 \\
\frac{1}{n}, & \mbox{if } a+2 \leq i \leq n
  \end{cases}.
\]
This proves that $P_D$ is a subpolytope of $P_{K_n}$. 
\end{proof}

\subsection{Lattice simplices associated to generalized cycles}

In \cite{BM17} the authors study the Laplacian simplex associated to the undirected cycle graph $C_n$, proving the following result.

\begin{theorem}[{\cite[Theorem 5.1]{BM17}}]
\label{thm:ref_cycles}
For $n \ge 3$, the simplex $T_{C_n}$ is reflexive if and only if $n$ is odd.
\end{theorem}

The rest of this section is aimed to generalize their result to the case of directed cycles. Note that in order to have reflexivity (or, in particular, to have one interior lattice point) we need the digraph to be strongly connected (Corollary~\ref{cor:0ininterior}). Therefore, all cycles we consider will always contain a cycle entirely oriented in one of the two possible directions and some additional edges directed in the opposite direction. Informally speaking, we define a cycle digraph to have all the edges pointing clockwise and some edges pointing counterclockwise.

\begin{definition}
\label{def:cycle}
Let $n \geq 3$. We say that a digraph $D$ on the vertex set $[n]$ is a \emph{cycle digraph} if, up to a relabeling of the vertices, $E(D) = \overrightarrow{E}(D) \cup \overleftarrow{E}(D)$, where
\begin{align*}
\overrightarrow{E}(D) &= \{(1,2),(2,3),\ldots,(n-1,n),(n,1)\},\\
\overleftarrow{E}(D) &\subseteq \{(n,n-1),(n-1,n-2),\ldots,(2,1),(1,n)\}.
\end{align*}
If such a relabeling exists, $D$ is completely determined by $\overleftarrow{E}(D)$, and we denote it by $D=C_n^S$, where $S \subseteq [n]$ is the set of the tails of the directed edges in $\overleftarrow{E}(D)$. As an example see Figure~\ref{fig:dir_cycle}.
\end{definition}

\begin{figure}[h]
\centering
\begin{tikzpicture}[->,>=stealth',shorten >=1pt,auto,node distance=2cm,
  thick,main node/.style={}]

  \node[main node] at (72:1)	(1) {1};
  \node[main node] at (0:1)		(2) {2};
  \node[main node] at (288:1)	(3) {3};
  \node[main node] at (216:1)	(4) {4};
  \node[main node] at (144:1)	(5) {5};

  \path[every node/.style={font=\sffamily\small}]
    (1) edge [bend left=20] (2)
		edge [bend left=20] (5)
    (2) edge [bend left=20] (3)
    (3) edge [bend left=20] (4)
		edge [bend left=20] (2)
    (4) edge [bend left=20] (5)
    (5) edge [bend left=20] (1);
\end{tikzpicture}
\caption{The cycle digraph $C_5^{1,3}$.}
\label{fig:dir_cycle}
\end{figure}
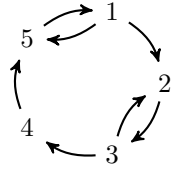

We first prove that, for most of the directed cycles, the associated Laplacian simplex has no other lattice points other than the origin and the vertices. Borrowing some terminology from the algebraic geometers, we call \emph{terminal Fano} a simplex with this property.

\begin{theorem}
\label{thm:terminal}
Let $D$ be any cycle digraph. $P_{D}$ is terminal Fano if and only if $D$ is not, up to a relabeling of the vertices, one of the following six exceptional directed cycles.
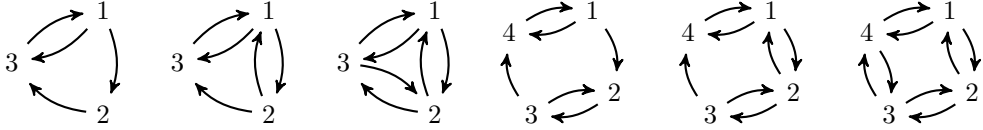
\begin{figure}[H]
\centering
\begin{tikzpicture}[->,>=stealth',shorten >=1pt,auto,node distance=2cm,
  thick,main node/.style={},scale=.8]

  \node[main node] at (60:1)	(1) {$1$};
  \node[main node] at (300:1)	(2) {$2$};
  \node[main node] at (180:1)	(3) {$3$};

  \path[every node/.style={font=\sffamily\small}]
    (1) edge [bend left=20] (2)
		edge [bend left=20] (3)
    (2) edge [bend left=20] (3)
    (3) edge [bend left=20] (1);
\end{tikzpicture}
\hspace*{.3cm}
\begin{tikzpicture}[->,>=stealth',shorten >=1pt,auto,node distance=2cm,
  thick,main node/.style={},scale=.8]

  \node[main node] at (60:1)	(1) {$1$};
  \node[main node] at (300:1)	(2) {$2$};
  \node[main node] at (180:1)	(3) {$3$};

  \path[every node/.style={font=\sffamily\small}]
    (1) edge [bend left=20] (2)
		edge [bend left=20] (3)
    (2) edge [bend left=20] (3)
		edge [bend left=20] (1)
    (3) edge [bend left=20] (1);
\end{tikzpicture}
\hspace*{.3cm}
\begin{tikzpicture}[->,>=stealth',shorten >=1pt,auto,node distance=2cm,
  thick,main node/.style={},scale=.8]

  \node[main node] at (60:1)	(1) {$1$};
  \node[main node] at (300:1)	(2) {$2$};
  \node[main node] at (180:1)	(3) {$3$};

  \path[every node/.style={font=\sffamily\small}]
    (1) edge [bend left=20] (2)
		edge [bend left=20] (3)
    (2) edge [bend left=20] (3)
		edge [bend left=20] (1)
    (3) edge [bend left=20] (1)
 		edge [bend left=20] (2);
\end{tikzpicture}
\hspace*{.3cm}
\begin{tikzpicture}[->,>=stealth',shorten >=1pt,auto,node distance=2cm,
  thick,main node/.style={},scale=.8]

  \node[main node] at (60:1)	(1) {$1$};
  \node[main node] at (330:1)	(2) {$2$};
  \node[main node] at (240:1)	(3) {$3$};
  \node[main node] at (150:1)	(4) {$4$};

  \path[every node/.style={font=\sffamily\small}]
    (1) edge [bend left=20] (2)
		edge [bend left=20] (4)
    (2) edge [bend left=20] (3)
    (3) edge [bend left=20] (4)
		edge [bend left=20] (2)
    (4) edge [bend left=20] (1);
\end{tikzpicture}
\hspace*{.3cm}
\begin{tikzpicture}[->,>=stealth',shorten >=1pt,auto,node distance=2cm,
  thick,main node/.style={},scale=.8]

  \node[main node] at (60:1)	(1) {$1$};
  \node[main node] at (330:1)	(2) {$2$};
  \node[main node] at (240:1)	(3) {$3$};
  \node[main node] at (150:1)	(4) {$4$};

  \path[every node/.style={font=\sffamily\small}]
    (1) edge [bend left=20] (2)
		edge [bend left=20] (4)
    (2) edge [bend left=20] (3)
		edge [bend left=20] (1)
    (3) edge [bend left=20] (4)
		edge [bend left=20] (2)
    (4) edge [bend left=20] (1);
\end{tikzpicture}
\hspace*{.3cm}
\begin{tikzpicture}[->,>=stealth',shorten >=1pt,auto,node distance=2cm,
  thick,main node/.style={},scale=.8]

  \node[main node] at (60:1)	(1) {$1$};
  \node[main node] at (330:1)	(2) {$2$};
  \node[main node] at (240:1)	(3) {$3$};
  \node[main node] at (150:1)	(4) {$4$};

  \path[every node/.style={font=\sffamily\small}]
    (1) edge [bend left=20] (2)
		edge [bend left=20] (4)
    (2) edge [bend left=20] (3)
		edge [bend left=20] (1)
    (3) edge [bend left=20] (4)
		edge [bend left=20] (2)
    (4) edge [bend left=20] (1)
		edge [bend left=20] (3);
\end{tikzpicture}
\end{figure}
\end{theorem}

\begin{proof}
We prove that, for $n \geq 5$, $P_{C_n^S}$ is terminal Fano for all $S \subseteq [n]$. The lower dimensional cases are checked individually, leading to the six exceptional cases above. For $1 \leq i \leq n$, we set $\bv_i=a_{i-1}\be_{i-1}+b_{i}\be_i-\be_{i+1}$ with $a_{j} \in \{-1,0\}$ and $b_{j}=1-a_{j} \in \{1,2\}$, where $a_0=a_n$, $\be_0=\be_n$ and $\be_{n+1}=\be_{1}$. Assume that $P_{C_n^S}$ is not terminal Fano.
Let $\bx=(x_1,\ldots,x_n)$ be a lattice point of $P_{C_n^S} \setminus (\{\bv_1,\ldots,\bv_n, {\bf 0} \})$ and set 
$\bx= \sum_{i=1}^{n}\lambda_i \bv_i$ with $0 \leq \lambda_1,\ldots,\lambda_n <1$ and $\lambda_1+\cdots+\lambda_n=1$.
Then one has $x_i=-\lambda_{i-1}+b_{i}\lambda_i+a_{i+1}\lambda_{i+1} \in \{-1,0,1\}$ for each $i$,
where $\lambda_0=\lambda_n$, $\lambda_{n+1}=\lambda_1$ and $a_{n+1}=a_1$.
If $x_2=-1$, then we obtain $a_{3}=-1$, $0<\lambda_{1},\lambda_{3} <1$ and $\lambda_j=0$ for any $j \neq 1,3$.
Since $x_{3}=2\lambda_{3}$, it follows that $\lambda_1=\lambda_3=1/2$ and $a_{1}=-1$.
However, $x_4=-1/2$, a contradiction.
Hence we have $x_i \in \{0,1\}$ for each $i$.
Since $\bx \neq \bzero$, we can assume that $x_3=1$.
Then one has $b_3=2$ and $\lambda_3 \geq 1/2$.
If $b_2=1$, then we obtain $\lambda_2=\lambda_3=1/2$.
However, $x_3=1/2$, a contradiction.
Hence one has $b_2=b_4=2$.
Moreover, it follows from  $\lambda_3 \geq 1/2$ that $x_2=x_4=0$.
Then it follows that $\lambda_1=3a_5\lambda_5+4\lambda_4-2 \leq 4 \lambda_4-2 \leq 0$.
Hence we obtain $\lambda_3=\lambda_4=1/2$.
However, $x_3=1/2$, a contradiction.
Therefore, $P_{C_n^S}$ is terminal Fano. 
\end{proof}

Now we characterize reflexivity for Laplacian simplices $P_{C_n^S}$, extending Theorem~\ref{thm:ref_cycles} by Braun--Meyer.

\begin{theorem}
\label{thm:ref_dir_cycles}
The Laplacian simplex $P_{C_n^S}$ associated to a cycle digraph $C_n^S$, is reflexive if and only if one of the following conditions is satisfied:
\begin{enumerate}
\item $S = \varnothing$, or
\item $S = [n]$ and $n=2$, or
\item $S = [n]$ and $n$ is odd, or
\item $\varnothing \subsetneq S \subsetneq [n]$, such that $k | c(D)$ for each integer $1 \leq k \leq K+1$, where $K$ is the longest chain of consecutive edges pointing counterclockwise, i.e. 
\[
K \coloneqq \max \{ j \mid \{a+1,\ldots,a+j\} \subseteq S, \mbox{ for some  } a \in [n]\},
\]
where, since $S \subsetneq [n]$, we have assumed without loss of generality, that $1 \notin S$.
\end{enumerate}
\end{theorem}
\begin{proof}
If $S$ satisfies $(1)$ or $(2)$, then it trivial to check that $P_{C_n^S}$ is reflexive. If $S$ satisfies $(3)$, then $P_{C_n^S}$ is reflexive by Theorem~\ref{thm:ref_cycles}. Suppose now that $S$ satisfies $(4)$. In particular we have assumed that $1 \notin S$. This implies vertex $n$ has exactly one spanning converging tree, i.e. $c_{n}=1$. 
As usual, $c_i$ denotes, the number of spanning trees which converge to vertex $i$. Then $\gcd(c_1, \ldots, c_n) = 1$, and $P_{C_n^S}$ is a weighted projective space by Proposition~\ref{prop:wps}.
For each vertex $i$ we denote by $K_i$ the length of the longest chain of consecutive edges pointing counterclockwise ending in $i$, i.e.
\[
K_i \coloneqq \max \{ j \mid \{i+1,\ldots,i+j\} \subseteq S\}, \qquad \mbox{for $1 \leq i \leq n$ }
\]
in particular $K_n=0$ and $K= \max \{K_i \mid i \in [n] \}$. Given $i \in [n]$, note that there are exactly $K_i + 1$ spanning trees converging to $i$. There are $K_i$,  having edge set
\[
\{ (j,j-1),\ldots,(i+1,i),(j+1,j+2),\ldots,(n-1,n),(n,1),\ldots,(i-1,i) \},
\]
for all $j \in \{i+1,\ldots,i+K_i\}$, plus an additional ``clockwise tree'' with edges
\[
\{(i+1,i+2),\ldots,(n-1,n),(n,1),\ldots,(i-1,i) \}.
\]
By Corollary~\ref{cor:refl}, $P_{C_n^S}$ is reflexive if and only if $c_i|c(D)$, for all $i \in [n]$. We conclude by noting that if $c_i > 1$ for some $i \in [n]$, then $c_{i+1}=c_i-1$, in particular $\{c_i \mid i \in [n] \} = \{ 1, \ldots, K+1\}$.
\end{proof}

We now have all the tools to completely characterize all the reflexive IDP simplices arising from cycle digraphs.

\begin{theorem}
\label{thm_IDP}
Let $C_n^S$ be a cycle digraph on $n$ vertices such that $P_{C_n^S}$ is reflexive. Then $P_{C_n^S}$ possesses the integer decomposition property if and only if $D$ satisfies one of the following conditions:
\begin{enumerate}
\item $S=\varnothing$, or
\item $D$ is, up to a relabeling of the vertices, one of the following directed cycles.
\end{enumerate}
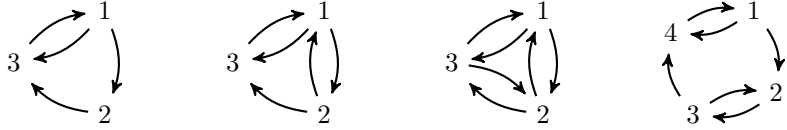
\begin{figure}[h]
\centering
\begin{tikzpicture}[->,>=stealth',shorten >=1pt,auto,node distance=2cm,
  thick,main node/.style={},scale=.8]

  \node[main node] at (60:1)	(1) {$1$};
  \node[main node] at (300:1)	(2) {$2$};
  \node[main node] at (180:1)	(3) {$3$};

  \path[every node/.style={font=\sffamily\small}]
    (1) edge [bend left=20] (2)
		edge [bend left=20] (3)
    (2) edge [bend left=20] (3)
    (3) edge [bend left=20] (1);
\end{tikzpicture}
\hspace*{1cm}
\begin{tikzpicture}[->,>=stealth',shorten >=1pt,auto,node distance=2cm,
  thick,main node/.style={},scale=.8]

  \node[main node] at (60:1)	(1) {$1$};
  \node[main node] at (300:1)	(2) {$2$};
  \node[main node] at (180:1)	(3) {$3$};

  \path[every node/.style={font=\sffamily\small}]
    (1) edge [bend left=20] (2)
		edge [bend left=20] (3)
    (2) edge [bend left=20] (3)
		edge [bend left=20] (1)
    (3) edge [bend left=20] (1);
\end{tikzpicture}
\hspace*{1cm}
\begin{tikzpicture}[->,>=stealth',shorten >=1pt,auto,node distance=2cm,
  thick,main node/.style={},scale=.8]

  \node[main node] at (60:1)	(1) {$1$};
  \node[main node] at (300:1)	(2) {$2$};
  \node[main node] at (180:1)	(3) {$3$};

  \path[every node/.style={font=\sffamily\small}]
    (1) edge [bend left=20] (2)
		edge [bend left=20] (3)
    (2) edge [bend left=20] (3)
		edge [bend left=20] (1)
    (3) edge [bend left=20] (1)
 		edge [bend left=20] (2);
\end{tikzpicture}
\hspace*{1cm}
\begin{tikzpicture}[->,>=stealth',shorten >=1pt,auto,node distance=2cm,
  thick,main node/.style={},scale=.8]

  \node[main node] at (60:1)	(1) {$1$};
  \node[main node] at (330:1)	(2) {$2$};
  \node[main node] at (240:1)	(3) {$3$};
  \node[main node] at (150:1)	(4) {$4$};

  \path[every node/.style={font=\sffamily\small}]
    (1) edge [bend left=20] (2)
		edge [bend left=20] (4)
    (2) edge [bend left=20] (3)
    (3) edge [bend left=20] (4)
		edge [bend left=20] (2)
    (4) edge [bend left=20] (1);
\end{tikzpicture}
\end{figure}
\end{theorem}
\begin{proof}
If $S=\varnothing$ then $C_n^S$ is always a reflexive IDP simplex. If $S=[n]$, from Theorem~\ref{thm:ref_cycles}, $P_{C_n^{[n]}}$ is reflexive if and only if $n$ is odd. In this case it is known that $P_{C_n^{[n]}}$ is IDP if and only if $n=3$ \cite[Corollary~5.11]{BM17}. Now, assume that $\varnothing \neq S \neq [n]$ and $P_{C_n^S}$ is IDP. We use the same notation introduced in Theorem~\ref{thm:ref_dir_cycles}. Then we can assume that $c_1=1,c_2=K+1, c_3 = K, \ldots,c_{K+1}=2$. Set $\bq=(c_2,\ldots,c_n)$. It follows that $P_{C_n^S}$ is unimodularly equivalent to $\Delta_{(1,\bq)}$. By Lemma~\ref{lem:BDS_IDP}, we know that, for each $2 \leq j \leq n$
\begin{equation}
\label{eq:IDP}
\frac{1}{c_j}+\sum_{i\neq j}\left\{\frac{c_i}{c_j}\right\}=1.
\end{equation}
But, if $K \geq 3$, by \eqref{eq:IDP} we get
\[
\dfrac{1}{K+1}+\sum_{i=3}^{n}\left\{\dfrac{c_i}{K+1}\right\} \geq \dfrac{1}{K+1}+\dfrac{K-1}{K+1}+\dfrac{K}{K+1}>1,
\]
so $K$ must satisfy $1 \leq K \leq 2$. By applying \eqref{eq:IDP} in these cases one gets $n \leq 4$. By checking all the cycle digraph having up to four vertices, we get the exceptions represented above.
\end{proof}

As an application of the tools developed in this section, we build a special family of cycle digraphs whose Laplacian polytopes are reflexive and have non unimodal $h^*$-vectors.

\begin{theorem}
\label{thm:non_unim}
Let $\alpha, \beta, k \in \Z_{>0}$ such that $\alpha \le \beta \le k-1$ and $\alpha + \beta \le k+1$.
Let $D=C_n^S$ be the cycle digraph, with $n \coloneqq 6(k+1) - 2 \alpha - \beta$, and $S \coloneqq S_1 \cup S_2 \cup S_3$, with
\begin{align*}
& S_1 \coloneqq \{1 + 3h \mid 0 \leq h \leq \alpha - 1\}, \\
& S_2 \coloneqq \{2 + 3h \mid 0 \leq h \leq \alpha - 1\}, \\
& S_3 \coloneqq \{ 3\alpha + 1 + 2h \mid 0 \leq h \leq \beta-\alpha - 1\}.
\end{align*}
Then $P_{D}$ is a reflexive simplex of dimension $6(k+1) - 2 \alpha - \beta - 1$  with symmetric and nonunimodal $h^*$-vector
\[
(
\underbrace{1,\ldots,1}_{2(k+1)-\alpha} ,
\underbrace{2,\ldots,2}_{\alpha} ,
\underbrace{1,\ldots,1}_{(k+1)-\alpha-\beta} ,
\underbrace{2,\ldots,2}_{\beta} ,
\underbrace{1,\ldots,1}_{(k+1)-\alpha-\beta} ,
\underbrace{2,\ldots,2}_{\alpha} ,
\underbrace{1,\ldots,1}_{2(k+1)-\alpha} )
\]

\end{theorem}

\begin{proof} An example of the digraph in the statement is represented in Figure~\ref{fig:non_unim}. 
The digraph has no more than two consecutive vertices with outdegree two, so the number of spanning trees converging to each of the vertices of $D$ is at most three. 
Specifically,
\[
c_i = 
\begin{cases}
3, & \mbox{if } i \in S_1,\\
2, & \mbox{if } i \in S_2 \cup S_3,\\
1, & \mbox{if } i \in [n] \setminus S.
\end{cases}
\]
Since each $c_i$ divides $c(D) = \sum_{i = 1 }^n c_i = 6(k+1)$, then $P_D$ is reflexive by Theorem~\ref{thm:ref_dir_cycles}. 
Now we use Theorem~\ref{thm:hstarq} to describe its $h^*$-polynomial. 
In particular \[
h^*(P_D,z)=\sum_{b=0}^{c(D)-1} x^{w(b)}, \quad \text{with} \; w(b)=b-\sum_{i=1}^{n} \left \lfloor \frac{c_i b}{6(k+1)} \right \rfloor.
\]
In our case this becomes
\[
w(b) = b - \alpha  \left \lfloor \frac{b}{2(k+1)}\right \rfloor - \beta  \left \lfloor \frac{b}{3(k+1)} \right \rfloor,
\]
which yields
\[
w(b) = \left\{\begin{alignedat}{5}
&b, 					&&	\mbox{if } \quad 	&0 			&\leq 	b \leq 2(k+1)-1, \\
&b-\alpha, 				&&	\mbox{if } 			&2(k+1) 	&\leq 	b \leq 3(k+1)-1, \\
&b-\alpha-\beta, 		&&	\mbox{if } 			&3(k+1) 	&\leq 	b \leq 4(k+1)-1, \\
&b-2\alpha-\beta, \quad &&	\mbox{if } 			&4(k+1) 	&\leq 	b \leq 6(k+1)-1.
\end{alignedat}\right.
\]
From this we deduce the $i$-th coefficient of the $h^*$-polynomial:
\[
h^*_i =
\begin{cases}
2, & \mbox{if } \qquad
	\begin{cases}
	2(k+1)-\alpha \leq i \leq 2(k+1)-1, \,\text{or}\\
	3(k+1)-\alpha -\beta \leq i \leq 3(k+1)-\alpha-1,\,\text{or}\\
	4(k+1)-2\alpha -\beta \leq i \leq 4(k+1)-\alpha-\beta-1;  \,\\
	\end{cases} \\
1, & \mbox{otherwise.} \\
\end{cases}
\]

\end{proof}

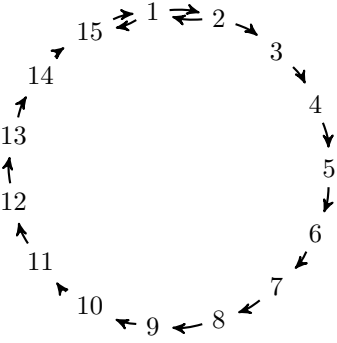
\begin{figure}
\centering
\begin{tikzpicture}[->,>=stealth',shorten >=1pt,auto,node distance=2cm,
  thick,main node/.style={},scale=.7]

  \node[main node] at (96:3)	(1) {$1$};
  \node[main node] at (72:3)	(2) {$2$};
  \node[main node] at (48:3)	(3) {$3$};
  \node[main node] at (24:3)	(4) {$4$};
  \node[main node] at (0:3)		(5) {$5$};
  \node[main node] at (336:3)	(6) {$6$};
  \node[main node] at (312:3)	(7) {$7$};
  \node[main node] at (288:3)	(8) {$8$};
  \node[main node] at (264:3)	(9) {$9$};
  \node[main node] at (240:3)	(10) {$10$};
  \node[main node] at (216:3)	(11) {$11$};
  \node[main node] at (192:3)	(12) {$12$};
  \node[main node] at (168:3)	(13) {$13$};
  \node[main node] at (144:3)	(14) {$14$};
  \node[main node] at (120:3)	(15) {$15$};

  \path[every node/.style={font=\sffamily\small}]
    (1) edge [bend left=10] (2)
		edge [bend left=10] (15)
    (2) edge [bend left=10] (3)
		edge [bend left=10] (1)
    (3) edge [bend left=10] (4)
    (4) edge [bend left=10] (5)
    (5) edge [bend left=10] (6)
    (6) edge [bend left=10] (7)
    (7) edge [bend left=10] (8)
    (8) edge [bend left=10] (9)
    (9) edge [bend left=10] (10)
    (10) edge [bend left=10] (11)
    (11) edge [bend left=10] (12)
    (12) edge [bend left=10] (13)
    (13) edge [bend left=10] (14)
    (14) edge [bend left=10] (15)
    (15) edge [bend left=10] (1);
\end{tikzpicture}
\caption{An example of the construction of Theorem~\ref{thm:non_unim}. In this case $\alpha=\beta=1$ and $k=2$. The Laplacian simplex associated to this digraph has $h^*$-polynomial $(1,1,1,1,1,2,1,2,1,2,1,1,1,1,1)$.}
\label{fig:non_unim}
\end{figure}

\section{Further Questions}
\label{sec:underlying}
Note that, in the case of undirected cycles studied by Braun-Meyer \cite{BM17}, the reflexivity is influenced by the number of vertices of the graph (Theorem~\ref{thm:ref_cycles}). On the other hand, when passing to the undirected case we discussed in Section~\ref{sec:4}, it is clear (from Theorem~\ref{thm:ref_dir_cycles}) that one can build reflexive Laplacian simplices starting from cycles of any length. This can be done by orienting a cycle in one of the two directions.

It is natural to wonder how the structure of the underlying simple graph $G_D$ of a digraph $D$ plays a role in determining the reflexivity of $P_D$ (see Section~\ref{sec:4} for the definition of underlying simple graph). We formalize this problem with the three questions below. We define an \textit{oriented graph} to be a simple digraph $D$ such that if there is an edge pointing from $i$ to $j$, then there is no edge pointing from $j$ to $i$.

\begin{question}
\label{q1}
For any simple graph $G$ on $[n]$, does there exist an oriented graph $D$ on $[n]$ such that $G_D=G$ and $P_D$ is a reflexive $(n-1)$-simplex? 
\end{question}
\begin{question}
\label{q2}
For any simple graph $G$ on $[n]$, does there exist a digraph $D$ on $[n]$ such that $G_D=G$ and $P_D$ is a reflexive $(n-1)$-simplex? 
\end{question}
\begin{question}
\label{q3}
For any simple graph $G$ on $[n]$, does there exist a simple digraph $D$ on $[n]$ such that $G_D=G$ and $P_D$ is a reflexive $(n-1)$-simplex? 
\end{question}

We have been able to answer negatively the first two questions with the following examples. Let $G_1$ be the following graph.
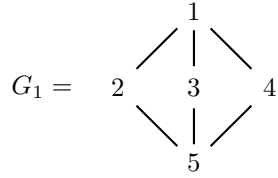
\begin{figure}[H]
	\centering
	\begin{tikzpicture}[>=stealth',shorten >=1pt,auto,node distance=2cm,
	thick,main node/.style={},scale=1]
	
	\node[main node] at (90:1)	(1) {$1$};
	\node[main node] at (180:1)	(2) {$2$};
	\node[main node] at (270:1)	(5) {$5$};
	\node[main node] at (0:1)	(4) {$4$};
	\node[main node] at (0:0)	(3) {$3$};
	\node[main node] at (180:2)	(6) {$G_1=$};
	
	\path[every node/.style={font=\sffamily\small}]
	(1) edge [bend right=0] (2)
		edge [bend left=0] (3)
		edge [bend left=0] (4)
	(5) edge [bend left=0] (2)
		edge [bend left=0] (3)
		edge [bend left=0] (4);
	\end{tikzpicture}
\end{figure}
\noindent Assume that $D_1$ is an orientation of $G_1$ such that $P_{D_1}$ is a reflexive $4$-simplex. Since $D_1$ is strongly connected, we may assume, without loss of generality, that the edges $(5,3),(3,1),(1,2),(2,5)$ are in $E(D_1)$. It follows that either $(1,4),(4,5)$ or $(5,4),(4,1)$ are in $E(D_1)$. In both cases, $P_{D_1}$ is not reflexive. In particular there are graphs that cannot be oriented to produce reflexive Laplacian simplices. Hence Question \ref{q1} is not true in general. However, $G_1$ is not a counterexample of Question \ref{q2}. Indeed, let $D_1'$ be the following simple digraph.
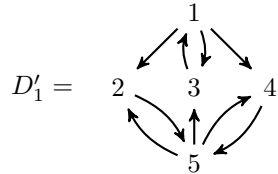
\begin{figure}[H]
	\centering
	\begin{tikzpicture}[->,>=stealth',shorten >=1pt,auto,node distance=2cm,
	thick,main node/.style={},scale=1]
	
	\node[main node] at (90:1)	(1) {$1$};
	\node[main node] at (180:1)	(2) {$2$};
	\node[main node] at (270:1)	(5) {$5$};
	\node[main node] at (0:1)	(4) {$4$};
	\node[main node] at (0:0)	(3) {$3$};
	\node[main node] at (180:2)	(6) {$D_1'=$};
	
	\path[every node/.style={font=\sffamily\small}]
	(1) edge [bend right=0] (2)
	edge [bend left=20] (3)
	edge [bend left=0] (4)
	(2) edge [bend left=20] (5)
	(3) edge [bend left=20] (1)
	(4) edge [bend left=20] (5)
	(5) edge [bend left=20] (2)
	edge [bend left=0] (3)
	edge [bend left=20] (4);
	\end{tikzpicture}
\end{figure}
\noindent Note that its underlying simple graph is still $G_1$, but $P_{D_1'}$ is reflexive. However, let $G_2$ be the following graph.
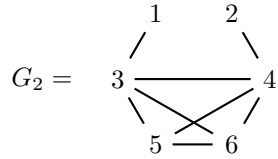
\begin{figure}[H]
	\centering
	\begin{tikzpicture}[>=stealth',shorten >=1pt,auto,node distance=2cm,
	thick,main node/.style={},scale=1]
	
	\node[main node] at (120:1)	(1) {$1$};
	\node[main node] at (60:1)	(2) {$2$};
	\node[main node] at (180:1)	(3) {$3$};
	\node[main node] at (0:1)		(4) {$4$};
	\node[main node] at (240:1)	(5) {$5$};
	\node[main node] at (300:1)	(6) {$6$};
	\node[main node] at (180:2)	(7) {$G_2=$};
	
	\path[every node/.style={font=\sffamily\small}]
	(1) edge [bend right=0] (3)
	(2) edge [bend right=0] (4)
	(3)	edge [bend left=0] (5)
	edge [bend left=0] (4)
	edge [bend left=0] (6)
	(4) edge [bend left=0] (5)
	edge [bend left=0] (6)
	(5)	edge [bend left=0] (6);
	\end{tikzpicture}
\end{figure}
\noindent
Note that there are finitely many possible directed simple graphs having it as underlying graph. A computer-assisted check shows that none of them produces a reflexive Laplacian simplex. Hence Question \ref{q2} is also not true in general. However, $G_2$ is not a counterexample of Question \ref{q3}. Indeed, let $D'_2$ be the following digraph (the label on an edge from $i$ to $j$, if present, represents the total number of edges from $i$ to $j$).
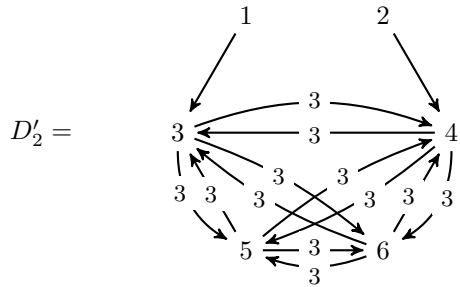
\begin{figure}[H]
	\centering
	\begin{tikzpicture}[->,>=stealth',shorten >=1pt,auto,node distance=2cm,
	thick,main node/.style={},scale=1.8]
	
	\node[main node] at (120:1)	(1) {$1$};
	\node[main node] at (60:1)	(2) {$2$};
	\node[main node] at (180:1)	(3) {$3$};
	\node[main node] at (0:1)		(4) {$4$};
	\node[main node] at (240:1)	(5) {$5$};
	\node[main node] at (300:1)	(6) {$6$};
	\node[main node] at (180:2)	(7) {$D_2'=$};
	
	\path[every node/.style={font=\sffamily\small,fill=white}]
	(1) edge [bend right=0] (3)
	(2) edge [bend right=0] (4)
	(3)	edge [bend right=30]	node[xshift=-.35cm,yshift=-.1cm] {${\tiny 3}$}  (5)
		edge [bend left=20]		node[xshift=-.0cm,yshift=-.2cm] {${\tiny 3}$}  (4)
		edge [bend left=10] 	node[xshift=-.3cm,yshift=-.2cm] {${\tiny 3}$}  (6)
	(4) edge [bend left=0] 		node[xshift=-.0cm,yshift=.2cm] {${\tiny 3}$}  (3)
		edge [bend left=10] 	node[xshift=-.0cm,yshift=.3cm] {${\tiny 3}$}  (5)
		edge [bend left=30] 	node[xshift=-.1cm,yshift=.3cm] {${\tiny 3}$}  (6)
	(5)	edge [bend left=0] 		node[xshift=.2cm,yshift=.2cm] {${\tiny 3}$}  (3)
		edge [bend left=10] 	node[xshift=.2cm,yshift=-.2cm] {${\tiny 3}$}  (4)
		edge [bend left=0] 		node[xshift=.0cm,yshift=-.2cm] {${\tiny 3}$}  (6)
	(6)	edge [bend left=10] 	node[xshift=.0cm,yshift=.3cm] {${\tiny 3}$}  (3)
		edge [bend left=0] 		node[xshift=.1cm,yshift=-.3cm] {${\tiny 3}$}  (4)
		edge [bend left=20] 	node[xshift=.0cm,yshift=.1cm] {${\tiny 3}$}  (5);
	\end{tikzpicture}
\end{figure}
\noindent
Then $P_{D'_2}$ is reflexive simplex. Question~\ref{q3} remains open.

\bibliographystyle{plain}

\begin{thebibliography}{10}

\bibitem{Bat94}
Victor~V. Batyrev.
\newblock Dual polyhedra and mirror symmetry for {C}alabi-{Y}au hypersurfaces
  in toric varieties.
\newblock {\em J. Algebraic Geom.}, 3(3):493--535, 1994.

\bibitem{BR15}
Matthias Beck and Sinai Robins.
\newblock {\em Computing the continuous discretely}.
\newblock Undergraduate Texts in Mathematics. Springer, New York, second
  edition, 2015.
\newblock Integer-point enumeration in polyhedra, With illustrations by David
  Austin.

\bibitem{Bra16}
Benjamin Braun.
\newblock Unimodality problems in {E}hrhart theory.
\newblock In {\em Recent trends in combinatorics}, volume 159 of {\em IMA Vol.
  Math. Appl.}, pages 687--711. Springer, [Cham], 2016.

\bibitem{BDS16}
Benjamin Braun, Robert Davis, and Liam Solus.
\newblock Detecting the integer decomposition property and ehrhart unimodality
  in reflexive simplices, 2016.
\newblock Preprint \arXiv{1608.01614}.

\bibitem{BM17}
Benjamin Braun and Marie Meyer.
\newblock Laplacian simplices, 2017.
\newblock Preprint \arXiv{1706.07085}.

\bibitem{MR1310575}
Francesco Brenti.
\newblock Log-concave and unimodal sequences in algebra, combinatorics, and
  geometry: an update.
\newblock In {\em Jerusalem combinatorics '93}, volume 178 of {\em Contemp.
  Math.}, pages 71--89. Amer. Math. Soc., Providence, RI, 1994.

\bibitem{Chu97}
Fan R.~K. Chung.
\newblock {\em Spectral graph theory}, volume~92 of {\em CBMS Regional
  Conference Series in Mathematics}.
\newblock Published for the Conference Board of the Mathematical Sciences,
  Washington, DC; by the American Mathematical Society, Providence, RI, 1997.

\bibitem{Con02}
Heinke Conrads.
\newblock Weighted projective spaces and reflexive simplices.
\newblock {\em Manuscripta Math.}, 107(2):215--227, 2002.

\bibitem{Dol82}
Igor Dolgachev.
\newblock Weighted projective varieties.
\newblock In {\em Group actions and vector fields ({V}ancouver, {B}.{C}.,
  1981)}, volume 956 of {\em Lecture Notes in Math.}, pages 34--71. Springer,
  Berlin, 1982.

\bibitem{Hib92}
Takayuki Hibi.
\newblock Dual polytopes of rational convex polytopes.
\newblock {\em Combinatorica}, 12(2):237--240, 1992.

\bibitem{Ian00}
A.~R. Iano-Fletcher.
\newblock Working with weighted complete intersections.
\newblock In {\em Explicit birational geometry of 3-folds}, volume 281 of {\em
  London Math. Soc. Lecture Note Ser.}, pages 101--173. Cambridge Univ. Press,
  Cambridge, 2000.

\bibitem{Kas13}
Alexander Kasprzyk.
\newblock Classifying terminal weighted projective space, 2013.
\newblock Preprint \arXiv{1304.3029}.

\bibitem{MHNH11}
Tetsushi Matsui, Akihiro Higashitani, Yuuki Nagazawa, Hidefumi Ohsugi, and
  Takayuki Hibi.
\newblock Roots of {E}hrhart polynomials arising from graphs.
\newblock {\em J. Algebraic Combin.}, 34(4):721--749, 2011.

\bibitem{OH99}
H.~Ohsugi and T.~Hibi.
\newblock A normal {$(0,1)$}-polytope none of whose regular triangulations is
  unimodular.
\newblock {\em Discrete Comput. Geom.}, 21(2):201--204, 1999.

\bibitem{OH98}
Hidefumi Ohsugi and Takayuki Hibi.
\newblock Normal polytopes arising from finite graphs.
\newblock {\em J. Algebra}, 207(2):409--426, 1998.

\bibitem{OH06}
Hidefumi Ohsugi and Takayuki Hibi.
\newblock Special simplices and {G}orenstein toric rings.
\newblock {\em J. Combin. Theory Ser. A}, 113(4):718--725, 2006.

\bibitem{MR1110850}
Richard~P. Stanley.
\newblock Log-concave and unimodal sequences in algebra, combinatorics, and
  geometry.
\newblock In {\em Graph theory and its applications: {E}ast and {W}est
  ({J}inan, 1986)}, volume 576 of {\em Ann. New York Acad. Sci.}, pages
  500--535. New York Acad. Sci., New York, 1989.

\bibitem{Sta99}
Richard~P. Stanley.
\newblock {\em Enumerative combinatorics. {V}ol. 2}, volume~62 of {\em
  Cambridge Studies in Advanced Mathematics}.
\newblock Cambridge University Press, Cambridge, 1999.
\newblock With a foreword by Gian-Carlo Rota and appendix 1 by Sergey Fomin.

\bibitem{TZ14}
Tuan Tran and G\"unter~M. Ziegler.
\newblock Extremal edge polytopes.
\newblock {\em Electron. J. Combin.}, 21(2):Paper 2.57, 16, 2014.

\bibitem{Vil98}
Rafael~H. Villarreal.
\newblock On the equations of the edge cone of a graph and some applications.
\newblock {\em Manuscripta Math.}, 97(3):309--317, 1998.

\end{thebibliography}

\end{document}